\documentclass[11pt,fleqn]{article}
\usepackage[utf8]{inputenc}
\usepackage[top=1in,bottom=1in,left=1in,right=1in]{geometry}
\usepackage{fouriernc}
\usepackage{authblk}
\usepackage[intlimits]{amsmath}
\usepackage{multicol}
\usepackage{amssymb}
\usepackage{graphicx}
\usepackage{amsthm}
\usepackage{appendix}

\newcommand*\pfqskip{8mu}
\catcode`,\active
\newcommand*\pfq{\begingroup
        \catcode`\,\active
        \def ,{\mskip\pfqskip\relax}%
        \dopfq
}
\catcode`\,12
\def\dopfq#1#2#3#4#5{%
        {}_{#1}F_{#2}\left(\genfrac..{0pt}{}{#3}{#4}\,\Big\rvert\,#5\right)%
        \endgroup
}

\catcode`,\active

\catcode`\,12

\catcode`,\active

\catcode`\,12

\newtheorem{proposition}{Proposition}[section]
\title{Two-variable $-1$ Jacobi polynomials}
\author[1]{Vincent X. Genest}
\author[1]{Jean-Michel Lemay}
\author[1]{Luc Vinet}
\author[2]{Alexei Zhedanov}
\affil[1]{Centre de recherches math\'ematiques, Universit\'e de Montr\'eal, P.O. Box 6128, Centre-ville Station, Montr\'eal, Canada, H3C 3J7}
\affil[2]{Donetsk Institute for Physics and Technology, Donetsk 340114, Ukraine}
\date{}

\begin{document}
\maketitle
\thispagestyle{empty}
\hrule
\begin{abstract}
\noindent
A two-variable generalization of the Big $-1$ Jacobi polynomials is introduced and characterized. These bivariate polynomials are constructed as a coupled product of two univariate Big $-1$ Jacobi polynomials. Their orthogonality measure is obtained. Their bispectral properties (eigenvalue equations and recurrence relations) are determined through a limiting process from the two-variable Big $q$-Jacobi polynomials of Lewanowicz and Wo\'zny. An alternative derivation of the weight function using Pearson-type equations is presented.\medskip

\\
\noindent \textbf{AMS classification numbers:} 33C50 
\end{abstract}
\hrule
\section{Introduction}
The purpose of this paper is to introduce and study a family of bivariate Big $-1$ Jacobi polynomials. These two-variable polynomials, which shall be denoted by $\mathcal{J}_{n,k}(x,y)$, depend on four real parameters $\alpha,\beta,\gamma,\delta$ such that $\alpha,\beta,\gamma>-1$, $\delta\neq 1$ and are defined as
\begin{align}
\label{Poly}
\mathcal{J}_{n,k}(x,y)=J_{n-k}\left(y;\alpha,2k+\beta+\gamma+1,(-1)^{k}\delta\right)\,\rho_{k}(y)\,J_{k}\left(\frac{x}{y};\gamma,\beta,\frac{\delta}{y}\right),
\end{align}
with $k=0,1,\ldots$ and $n=k,k+1,\ldots$, where
\begin{align*}
\rho_{k}(y)=
\begin{cases}
y^{k}\left(1-\frac{\delta^2}{y^2}\right)^{\frac{k}{2}}, & \text{$k$ even},
\\
y^{k}\left(1-\frac{\delta^2}{y^2}\right)^{\frac{k-1}{2}} \left(1+\frac{\delta}{y}\right), & \text{$k$ odd},
\end{cases}
\end{align*}
and where $J_{n}(x;a,b,c)$ denotes the one-variable Big $-1$ Jacobi polynomials \cite{Vinet_Zhedanov_2012-07} (see section 1.1). It will be shown that these polynomials are orthogonal with respect to a positive measure defined on the disjoint union of four triangular domains in the real plane. The polynomials $\mathcal{J}_{n,k}(x,y)$ will also be identified as a $q\rightarrow -1$ limit of the two-variable Big $q$-Jacobi polynomials introduced by Lewanowicz and Wo\'zny in \cite{Lewanowicz_Wozny_2010-01}, which generalize the bivariate little $q$-Jacobi polynomials introduced by Dunkl in \cite{Dunkl_1980}. The bispectral properties of the Big $-1$ Jacobi polynomials will be determined from this identification. The polynomials $\mathcal{J}_{n,k}(x,y)$ will be shown to satisfy an explicit vector-type three term recurrence relation and it will be seen that they are the joint eigenfunctions of a pair of commuting first order differential operators involving reflections. By solving the Pearson-type system of equations arising from the symmetrization of these differential/difference operators, the weight function for the polynomials $\mathcal{J}_{n,k}(x,y)$ will be recovered.  

The defining formula of the two-variable Big $-1$ Jacobi polynomials \eqref{Poly} is reminiscent of the expressions found in \cite{Kwon_Lee_Littlejohn_2001} for the Krall-Sheffer polynomials  \cite{Krall_Sheffer_1967}, which, as shown in \cite{Harnad-2001}, are directly related to two-dimensional superintegrable systems on spaces with constants curvature (see \cite{Miller_Post_Winter_2013} for a review of superintegrable systems). The polynomials $\mathcal{J}_{n,k}(x,y)$ do not belong to the Krall-Sheffer classification, as they will be seen to obey \emph{first order} differential equations with reflections. The results of \cite{Harnad-2001} however suggest that the polynomials $\mathcal{J}_{n,k}(x,y)$ could be related to two-dimensional integrable systems with reflections such as the ones recently considered in \cite{Genest_Ismail_Vinet_Zhedanov_2013, Genest_Vinet_Zhedanov_2013, Genest_CMP_2014, Genest_2015_1}. This fact motivates our examination of the polynomials $\mathcal{J}_{n,k}(x,y)$.

\subsection{The Big $-1$ Jacobi polynomials}
Let us now review some of the properties of the Big $-1$ Jacobi polynomials which shall be needed in the following. The Big $-1$ Jacobi polynomials, denoted by $J_{n}(x;a,b,c)$, were introduced in \cite{Vinet_Zhedanov_2012-07} as a $q=-1$ limit of the Big $q$-Jacobi polynomials \cite{Koekoek-2010}. They are part of the Bannai-Ito scheme of $-1$ orthogonal polynomials \cite{Genest-2013-02-1, Genest-2013-09-02, Tsujimoto-2012-03, Vinet-2011-01}. They are defined by
\begin{align}
\label{Jacobi}
J_{n}(x;a,b,c)=
\begin{cases}
\pfq{2}{1}{-\frac{n}{2},\frac{n+a+b+2}{2}}{\frac{a+1}{2}}{\frac{1-x^2}{1-c^2}}+\frac{n(1-x)}{(1+c)(a+1)}\; \pfq{2}{1}{1-\frac{n}{2}, \frac{n+a+b+2}{2}}{\frac{a+3}{2}}{\frac{1-x^2}{1-c^2}}, & \text{$n$ even},
\\[3mm]
\pfq{2}{1}{-\frac{n-1}{2}, \frac{n+a+b+1}{2}}{\frac{a+1}{2}}{\frac{1-x^2}{1-c^2}}-\frac{(n+a+b+1)(1-x)}{(1+c)(a+1)} \; \pfq{2}{1}{-\frac{n-1}{2}, \frac{n+a+b+3}{2}}{\frac{a+3}{2}}{\frac{1-x^2}{1-c^2}},& \text{$n$ odd},
\end{cases}
\end{align}
where ${}_{2}F_{1}$ is the standard Gauss hypergeometric function \cite{Andrews_Askey_Roy_1999}; when no confusion can arise, we shall simply write $J_{n}(x)$ instead of $J_{n}(x;a,b,c)$. The polynomials \eqref{Jacobi} satisfy the recurrence relation
\begin{align*}
x\,J_{n}(x)=A_{n}\, J_{n+1}(x)+(1-A_{n}-C_{n})\,J_{n}(x)+ C_{n}\,J_{n-1}(x),
\end{align*}
with coefficients
\begin{align*}
A_{n}=
\begin{cases}
\frac{(n+a+1)(c+1)}{2n+a+b+2}, & \text{$n$ even},
\\
\frac{(1-c)(n+a+b+1)}{2n+a+b+2}, & \text{$n$ odd},
\end{cases}
\qquad 
C_{n}=
\begin{cases}
\frac{n(1-c)}{2n+a+b}, & \text{$n$ even},
\\
\frac{(n+b)(1+c)}{2n+a+b}, & \text{$n$ odd}.
\end{cases}
\end{align*}
It can be seen that for $a,b>-1$ and $|c|\neq 1$ the polynomials $J_{n}(x)$ are positive-definite. The Big $-1$ Jacobi polynomials satisfy the eigenvalue equation
\begin{align}
\label{Eigen-Uni}
\mathcal{L} J_{n}(x)=\left\{(-1)^{n}\left(n+a/2+b/2+1/2\right)\right\} J_{n}(x),
\end{align}
where $\mathcal{L}$ is the most general first-order differential operator with reflection preserving the space of polynomials of a given degree. This operator has the expression
\begin{align*}
\mathcal{L}=\left[\frac{(x+c)(x-1)}{x}\right]\partial_{x}R+\left[\frac{c}{2x^2}+\frac{c a-b}{2x}\right](R-\mathbb{I})+\left[\frac{a+b+1}{2}\right]R,
\end{align*}
where $R$ is the reflection operator, i.e. $Rf(x)=f(-x)$, and $\mathbb{I}$ stands for the identity. The orthogonality relation of the Big $-1$ Jacobi polynomials is as follows. For $|c|<1$, one has 
\begin{align}
\label{Ortho-1}
\int_{\mathcal{C}}J_{n}(x;a,b,c)\,J_{m}(x;a,b,c) \;\omega(x;a,b,c)\;\mathrm{d}x=\left[\frac{(1-c^2)^{\frac{a+b+2}{2}}}{(1+c)}\right] h_{n}(a,b)\,\delta_{nm},
\end{align}
where the interval is $\mathcal{C}=[-1,-|c|]\cup [|c|,1]$ and the weight function reads
\begin{align}
\label{Weight-1}
\omega(x;a,b,c)=\theta(x)\,(1+x)\,(x-c)\,(x^2-c^2)^{\frac{b-1}{2}}\,(1-x^2)^{\frac{a-1}{2}},
\end{align}
with $\theta(x)$ is the sign function. The normalization factor $h_{n}$ is given by
\begin{align}
\label{hn}
h_{n}(a,b)=
\begin{cases}
\frac{2\;\Gamma\left(\frac{n+b+1}{2}\right)\Gamma\left(\frac{n+a+3}{2}\right) \left(\frac{n}{2}\right)!}{(n+a+1)\;\Gamma\left(\frac{n+a+b+2}{2}\right)\left(\frac{a+1}{2}\right)_{\frac{n}{2}}^2}, & \text{$n$ even},
\\[3mm]
\frac{(n+a+b+1)\;\Gamma\left(\frac{n+b+2}{2}\right)\Gamma \left(\frac{n+a+2}{2}\right) \left(\frac{n-1}{2}\right)!}{2\,\Gamma\left(\frac{n+a+b+3}{2}\right) \left(\frac{a+1}{2}\right)_{\frac{n+1}{2}}^2}, & \text{$n$ odd},
\end{cases}
\end{align}
where $(a)_{n}$ stands for the Pochhammer symbol \cite{Andrews_Askey_Roy_1999}. For $|c|>1$, one has
\begin{align}
\label{Ortho-Tilde}
\int_{\widetilde{\mathcal{C}}}J_{n}(x;a,b,c)\,J_{m}(x;a,b,c) \;\widetilde{\omega}(x;a,b,c)\;\mathrm{d}x=\left[\frac{\theta(c)(c^2-1)^{\frac{a+b+2}{2}}}{1+c}\right]\widetilde{h}_{n}(a,b)\,\delta_{nm},
\end{align}
where the interval is $\widetilde{\mathcal{C}}=[-|c|,-1]\cup [1,|c|]$ and the weight function reads
\begin{align}
\label{Weight-Tilde}
\widetilde{\omega}(x;a,b,c)=\theta(c\,x)\,(1+x)\,(c-x)\,(c^2-x^2)^{\frac{b-1}{2}}\,(x^2-1)^{\frac{a-1}{2}}.
\end{align}
In this case the normalization factor has the expression 
\begin{align}
\label{Hn-Tilde}
\widetilde{h}_{n}(a,b)=
\begin{cases}
\frac{2\;\Gamma\left(\frac{n+b+1}{2}\right)\Gamma\left(\frac{n+a+3}{2}\right)\left(\frac{n}{2}\right)! }{(n+a+1) \Gamma\left(\frac{n+a+b+2}{2}\right)\,\left(\frac{a+1}{2}\right)_{\frac{n}{2}}^2}, & \text{$n$ even},
\\[3mm]
\frac{(n+a+b+1)\,\Gamma\left(\frac{n+b+2}{2}\right)\Gamma\left(\frac{n+a+2}{2}\right)\left(\frac{n-1}{2}\right)!}{2\;\Gamma\left(\frac{n+a+b+3}{2}\right)\,\left(\frac{a+1}{2}\right)_{\frac{n+1}{2}}^2} & \text{$n$ odd}.
\end{cases}
\end{align}
The normalization factors $h_{n}$ and $\widetilde{h}_{n}$ were not derived in \cite{Vinet_Zhedanov_2012-07}. They have been obtained here using the orthogonality relation for the Chihara polynomials provided in \cite{Genest-2013-09-02} and the fact that the Big $-1$ Jacobi polynomials are related to the latter by a Christoffel transformation. The details of this derivation are presented in appendix A.
\section{Orthogonality of the two-variable Big $-1$ Jacobi polynomials}
We now prove the orthogonality property of the two-variable Big $-1$ Jacobi polynomials.
\begin{proposition}
Let $\alpha, \beta,\gamma>-1$ and $|\delta|<1$. The two-variable Big $-1$ Jacobi polynomials defined by \eqref{Poly} satisfy the orthogonality relation
\begin{align}
\label{Ortho-2Var}
\int_{D_{y}}\int_{D_{x}}\;\mathcal{J}_{n,k}(x,y)\;\mathcal{J}_{m,\ell}(x,y)\;W(x,y)\;\mathrm{d}x\;\mathrm{d}y=H_{nk}
\;\delta_{k\ell}\delta_{nm},
\end{align}
with respect to the weight function
\begin{align}
\label{Weight-2Var}
W(x,y)=\theta(x\,y)|y|^{\beta+\gamma}(1+y)\left(1+\frac{x}{y}\right)\left(\frac{x-\delta}{y}\right)(1-y^2)^{\frac{\alpha-1}{2}}\left(1-\frac{x^2}{y^2}\right)^{\frac{\gamma-1}{2}}\left(\frac{x^2-\delta^2}{y^2}\right)^{\frac{\beta-1}{2}}.
\end{align}
The integration domain is prescribed by
\begin{align}
\label{Domain}
D_{x}=[-|y|,-|\delta|]\cup [|\delta|, |y|], \qquad D_{y}=[-1,-|\delta|]\cup[|\delta|,1],
\end{align}
and the normalization factor $H_{nk}$ has the expression
\begin{align*}
H_{nk}=\left[\frac{(1-\delta^2)^{\frac{2k+\alpha+\beta+\gamma+3}{2}}}{(1+(-1)^{k}\delta)}\right]\;h_{k}(\gamma,\beta)\;h_{n-k}(\alpha,2k+\gamma+\beta+1),
\end{align*}
where $h_{n}(a,b)$ is given by \eqref{hn}.
\end{proposition}
\begin{proof}
We proceed by a direct calculation. We denote the orthogonality integral by
\begin{align*}
I=\int_{D_{y}}\int_{D_{x}}\;\mathcal{J}_{n,k}(x,y)\;\mathcal{J}_{m,\ell}(x,y)\;W(x,y)\;\mathrm{d}x\;\mathrm{d}y.
\end{align*}
Upon using the expressions \eqref{Poly} and \eqref{Weight-2Var} in the above, one writes
\begin{multline*}
I
=\int_{D_{y}} J_{n-k}(y;\alpha,2k+\gamma+\beta+1,(-1)^{k}\delta)\;\;J_{m-\ell}(y,\alpha,2k+\gamma+\beta+1,(-1)^{k}\delta)
\\
\times \left[\rho_{k}(y)\rho_{\ell}(y)|y|^{\beta+\gamma}(1+y)(1-y^2)^{\frac{\alpha-1}{2}}\right]\;\mathrm{d}y\;\;
\\
\times \int_{D_{x}} J_{k}\left(\frac{x}{y}; \gamma,\beta,\frac{\delta}{y}\right)J_{\ell}\left(\frac{x}{y}; \gamma,\beta,\frac{\delta}{y}\right)
\left[\theta\left(\frac{x}{y}\right)\left(1+\frac{x}{y}\right)\left(\frac{x-\delta}{y}\right)\left(1-\frac{x^2}{y^2}\right)^{\frac{\gamma-1}{2}}\left(\frac{x^2-\delta^2}{y^2}\right)^{\frac{\beta-1}{2}}\right]\;\mathrm{d}x.
\end{multline*}
The integral over $D_{x}$ is directly evaluated using the change of variables $u=x/y$ and comparing with the orthogonality relation \eqref{Ortho-1}. The result is thus
\begin{multline*}
I=h_{k}(\gamma,\beta)\;\delta_{k\ell}\times
\int_{D_y}J_{n-k}(y;\alpha,2k+\gamma+\beta+1,(-1)^{k}\delta)\;\;J_{m-\ell}(y,\alpha,2k+\gamma+\beta+1,(-1)^{k}\delta)
\\
\times \rho_{k}(y)\;\rho_{\ell}(y) \left[\theta(y)\;(1+y)\;(1-y^2)^{\frac{\alpha-1}{2}}\;\frac{(y^2-\delta^2)^{\frac{2+\gamma+\beta}{2}}}{(y+\delta)}\right]\;\mathrm{d}y.
\end{multline*}
Assuming that $k=\ell$ is an even integer, the integral takes the form
\begin{align*}
I=h_{k}(\gamma,\beta)\delta_{k\ell}\times \int_{D_y}J_{n-k}(y;\alpha,2k+\gamma+\beta+1,\delta)\;\;J_{m-\ell}(y,\alpha,2k+\gamma+\beta+1,\delta)
\\
\times (y^2-\delta^2)^{k} \left[\theta(y)\;(1+y)\;(1-y^2)^{\frac{\alpha-1}{2}}\;\frac{(y^2-\delta^2)^{\frac{2+\gamma+\beta}{2}}}{(y+\delta)}\right]\;\mathrm{d}y,
\end{align*}
which in view of \eqref{Ortho-1} yields
\begin{align*}
I=h_{k}(\gamma,\beta)h_{n-k}(\alpha,2k+\gamma+\beta+1)\left[\frac{(1-\delta^2)^{\frac{2k+\alpha+\beta+\gamma+3}{2}}}{1+\delta}\right]\delta_{k\ell}\delta_{mn}.
\end{align*}
Assuming that $k=\ell$ is an odd integer, the integral takes the form
\begin{align*}
I=h_{k}(\gamma,\beta)\delta_{k\ell}\times \int_{D_y}J_{n-k}(y;\alpha,2k+\gamma+\beta+1,-\delta)\;\;J_{m-\ell}(y,\alpha,2k+\gamma+\beta+1,-\delta)
\\
\times (y^2-\delta^2)^{k-1}(y+\delta)^{2} \left[\theta(y)\;(1+y)\;(1-y^2)^{\frac{\alpha-1}{2}}\;\frac{(y^2-\delta^2)^{\frac{2+\gamma+\beta}{2}}}{(y+\delta)}\right]\;\mathrm{d}y,
\end{align*}
which given \eqref{Ortho-1} gives
\begin{align*}
I=h_{k}(\gamma,\beta)h_{n-k}(\alpha,2k+\gamma+\beta+1)\left[\frac{(1-\delta^2)^{\frac{2k+\alpha+\beta+\gamma+3}{2}}}{1-\delta}\right]\delta_{k\ell}\delta_{mn}.
\end{align*}
Upon combining the $k$ even and $k$ odd cases, one finds \eqref{Ortho-2Var}. This completes the proof.
\end{proof}
It is not difficult to see that the region \eqref{Domain} corresponds to the disjoint union of four triangular domains. The $|\delta|=1/5$ case is illustrated in Figure 1.
\begin{center}
\scalebox{0.3}{\includegraphics{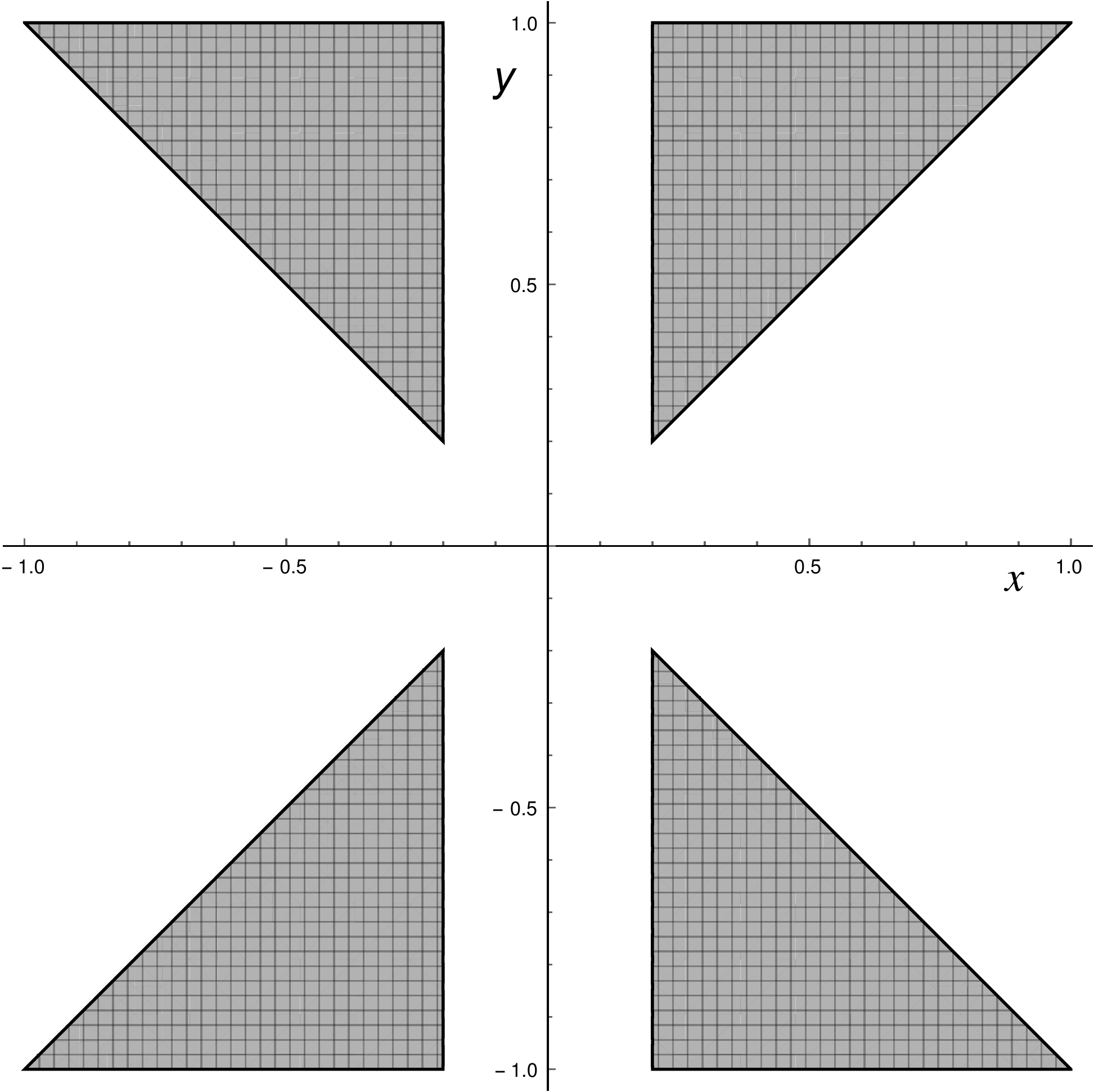}}
\\
\footnotesize
Figure 1: Orthogonality region for $|\delta|=1/5$
\end{center} 
For $\alpha,\beta,\gamma>-1$ and $|\delta|<1$, it can be verified that the weight function \eqref{Weight-2Var} is positive on \eqref{Domain}. The orthogonality relation for $|\delta|>1$ can be obtained in a similar fashion. The result is as follows.
\begin{proposition}
Let $\alpha, \beta,\gamma>-1$ and $|\delta|>1$. The two-variable Big $-1$ Jacobi polynomials defined by \eqref{Poly} satisfy the orthogonality relation
\begin{align*}
\int_{\widetilde{D}_{y}}\int_{\widetilde{D}_{x}}\;\mathcal{J}_{n,k}(x,y)\;\mathcal{J}_{m,\ell}(x,y)\;\widetilde{W}(x,y)\;\mathrm{d}x\;\mathrm{d}y=\widetilde{H}_{nk}
\;\delta_{k\ell}\delta_{nm},
\end{align*}
with respect to the weight function
\begin{align}
\label{Weight-2Var-2}
\widetilde{W}(x,y)=\theta(\delta\,x\,y) |y|^{\gamma+\beta}(1+y)\left(1+\frac{x}{y}\right)\left(\frac{\delta-x}{y}\right)(y^2-1)^{\frac{\alpha-1}{2}}\left(\frac{x^2}{y^2}-1\right)^{\frac{\gamma-1}{2}}\left(\frac{\delta^2-x^2}{y^2}\right)^{\frac{\beta-1}{2}}.
\end{align}
The integration domain is
\begin{align}
\label{Domain-2}
\widetilde{D}_{x}=[-|\delta|, -|y|]\cup [|y|,|\delta|],\qquad \widetilde{D}_{y}=[-|\delta|,1]\cup [1, |\delta|],
\end{align}
and the normalization factor is of the form
\begin{align*}
\widetilde{H}_{nk}=(-1)^{k}\theta(\delta)\left[\frac{(\delta^2-1)^{\frac{2k+\alpha+\beta+\gamma+3}{2}}}{1+(-1)^{k}\delta}\right]\widetilde{h}_{k}(\gamma,\beta)\,\widetilde{h}_{n-k}(\alpha,2k+\beta+\gamma+1),
\end{align*}
where $\widetilde{h}_{n}(a,b)$ is given by \eqref{Hn-Tilde}. 
\end{proposition}
\begin{proof}
Similar to proposition 2.1 using instead \eqref{Ortho-Tilde}, \eqref{Weight-Tilde} and \eqref{Hn-Tilde}.
\end{proof}
It can again be  seen that the weight function \eqref{Weight-2Var-2} is positive on the domain \eqref{Domain-2} provided that $\alpha,\beta,\gamma>-1$ and $|\delta|>1$. The orthogonality region defined by \eqref{Domain-2} again corresponds to the disjoint union of four triangular domains, as illustrated by Figure 2 for the case $|\delta|=3$.
\begin{center}
\scalebox{0.3}{\includegraphics{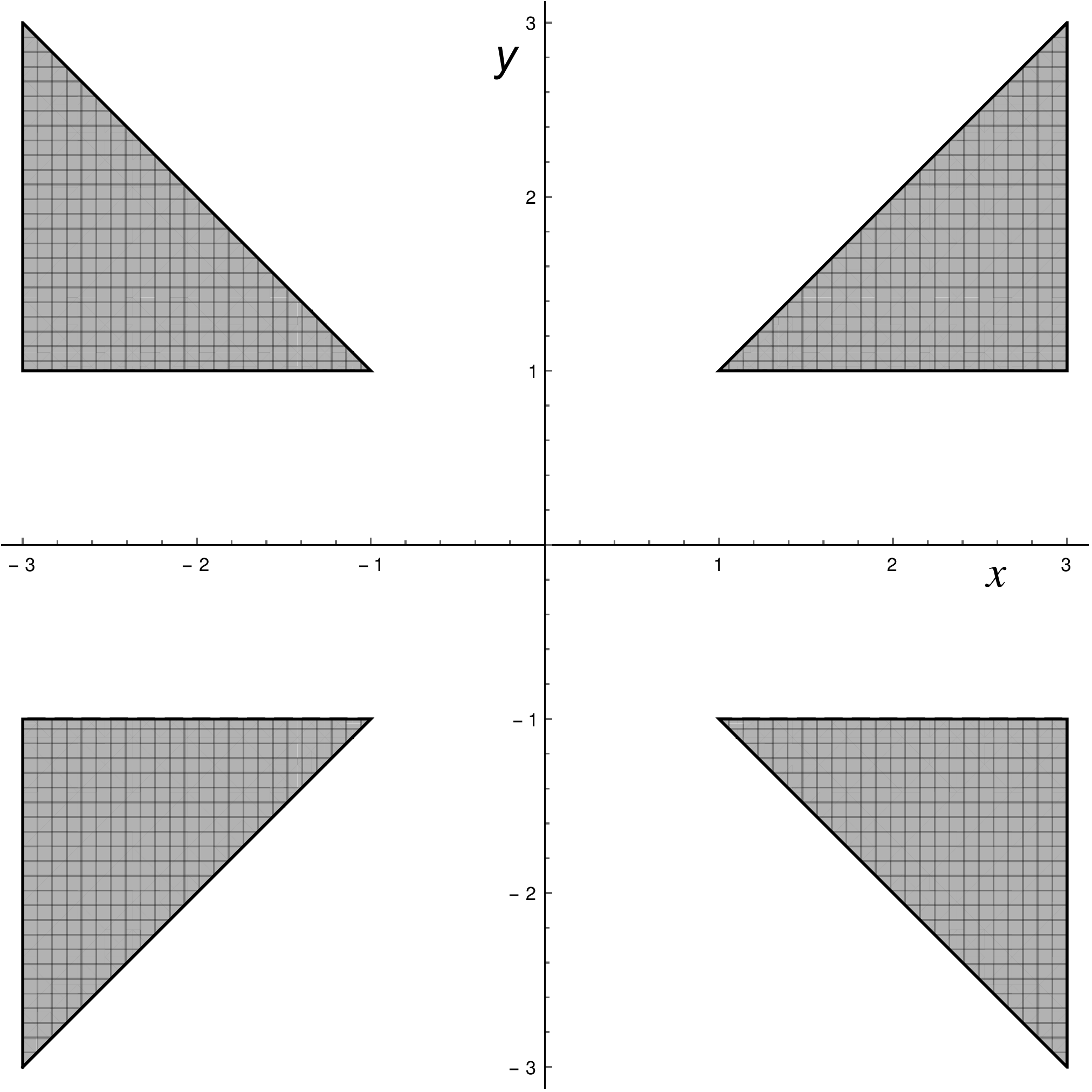}}
\\
\footnotesize
Figure 2: Orthogonality region for $|\delta|=3$
\end{center}

\subsection{A special case: the bivariate Little $-1$ Jacobi polynomials}
When $c=0$, the Big $-1$ Jacobi polynomials $J_{n}(x;a,b,c)$ defined by \eqref{Jacobi} reduce to the so-called Little $-1$ Jacobi polynomials $j_{n}(x;a,b)$ introduced in \cite{Vinet-2011-01}. These polynomials have the hypergeometric representation
\begin{align}
j_{n}(x;a,b)=
\begin{cases}
\pfq{2}{1}{-\frac{n}{2},\frac{n+a+b+2}{2}}{\frac{a+1}{2}}{1-x^2}+\frac{n(1-x)}{(a+1)}\; \pfq{2}{1}{1-\frac{n}{2}, \frac{n+a+b+2}{2}}{\frac{a+3}{2}}{1-x^2}, & \text{$n$ even},
\\[3mm]
\pfq{2}{1}{-\frac{n-1}{2}, \frac{n+a+b+1}{2}}{\frac{a+1}{2}}{1-x^2}-\frac{(n+a+b+1)(1-x)}{(a+1)} \; \pfq{2}{1}{-\frac{n-1}{2}, \frac{n+a+b+3}{2}}{\frac{a+3}{2}}{1-x^2},& \text{$n$ odd}.
\end{cases}
\end{align}
Taking $\delta=0$ in \eqref{Poly} leads to the following definition for the two-variable Little $-1$ Jacobi polynomials:
\begin{align}
\label{2-Var-Little}
q_{n,k}(x,y)=j_{n-k}(y;\alpha,2k+\beta+\gamma+1)\;y^{k}\;j_{k}\left(\frac{x}{y};\gamma,\beta\right),\quad k=0,1,2\ldots,\quad n=k,k+1,\ldots.
\end{align}
It is seen from \eqref{2-Var-Little} that the two-variable Little $-1$ Jacobi polynomials have the structure corresponding to one of the methods to construct bivariate orthogonal polynomials systems proposed by Koornwinder in \cite{Koornwinder-1975}. For the polynomials \eqref{2-Var-Little}, the weight function, which can be obtained by taking $\delta=0$ in either \eqref{Weight-2Var} or \eqref{Weight-2Var-2}, can also be recovered using the general scheme given in \cite{Koornwinder-1975}. For $\delta=0$, the region \eqref{Domain} reduces to two vertically opposite triangles.
\section{Bispectrality of the bivariate Big $-1$ Jacobi polynomials}
In this section, the two-variable Big $-1$ Jacobi polynomials $\mathcal{J}_{n,k}(x,y)$ are shown be the joint eigenfunctions of a pair of first-order differential operators involving reflections. Their recurrence relations are also derived. The results are obtained through a limiting process from the corresponding properties of the two-variable Big $q$-Jacobi polynomials introduced by Lewanowicz and Wo\'zny \cite{Lewanowicz_Wozny_2010-01}.
\subsection{Bivariate Big $q$-Jacobi polynomials}
Let us review some of the properties of the bivariate $q$-polynomials introduced in \cite{Lewanowicz_Wozny_2010-01}. The two-variable Big $q$-Jacobi polynomials, denoted $\mathcal{P}_{n,k}(x,y;a,b,c,d;q)$ are defined as
\begin{align}
\label{Biv-Q-Jacobi}
\mathcal{P}_{n,k}(x,y;a,b,c,d;q)=P_{n-k}(y;a,bc q^{2k+1}, dq^{k};q)\;y^{k}\;\left(\frac{d q}{y}; q\right)_{k}\;P_{k}\left(\frac{x}{y}; c, b,\frac{d}{y};q\right),
\end{align}
where $(a;q)_{n}$ stands for the $q$-Pochhammer symbol \cite{Gasper-2004} and where $P_{n}(x;a,b,c;q)$ are the Big $q$-Jacobi polynomials \cite{Koekoek-2010}. The two-variable Big $q$-Jacobi polynomials satisfy the eigenvalue equation \cite{Lewanowicz_Wozny_2010-01}
\begin{align}
\label{Eigen-q}
\Omega\,\mathcal{P}_{n,k}(x,y)=\left[\frac{q^{1-n}(q^{n}-1)(abcq^{n+2}-1)}{(q-1)^2}\right]\,\mathcal{P}_{n,k}(x,y),
\end{align}
where $\Omega$ is the $q$-difference operator 
\begin{multline*}
\Omega=(x-dq)(x-acq^2)\mathbf{D}_{q,x}\mathbf{D}_{q^{-1},x}+(y-aq)(y-dq)\mathbf{D}_{q,y}\mathbf{D}_{q^{-1},y}
\\
+q^{-1}(x-dq)(y-aq)\mathbf{D}_{q^{-1},x}\mathbf{D}_{q^{-1},y}+acq^{3}(bx-d)(y-1)\mathbf{D}_{q,x}\mathbf{D}_{q,y}
\\
+\frac{(abcq^{3}-1)(x-1)-(acq^{2}-1)(dq-1)}{q-1}\mathbf{D}_{q,x}+
\frac{(abcq^{3}-1)(y-1)-(aq-1)(dq-1)}{q-1}\mathbf{D}_{q,y},
\end{multline*}
and where $\mathbf{D}_{q,x}$ stands for the $q$-derivative
\begin{align*}
\mathbf{D}_{q,x}f(x,y)=\frac{f(qx,y)-f(x,y)}{x(q-1)}.
\end{align*}
The bivariate Big $q$-Jacobi polynomials also satisfy the pair of recurrence relations \cite{Lewanowicz_Wozny_2010-01}
\begin{align*}
y\,\mathcal{P}_{n,k}(x,y)=a_{nk}\,\mathcal{P}_{n+1,k}(x,y)+b_{nk}\,\mathcal{P}_{n,k}(x,y)+c_{nk}\,\mathcal{P}_{n-1,k}(x,y),
\end{align*}
\vspace{-2.5\baselineskip}
\begin{multline}
\label{Recurrence-q}
x\,\mathcal{P}_{n,k}(x,y)=e_{nk}\,\mathcal{P}_{n+1,k-1}(x,y)+f_{nk}\,\mathcal{P}_{n+1,k}(x,y)+g_{nk}\,\mathcal{P}_{n+1,k+1}(x,y)
\\ 
+r_{nk}\,\mathcal{P}_{n,k-1}(x,y)+s_{nk}\,\mathcal{P}_{n,k}(x,y)+t_{nk}\,\mathcal{P}_{n,k+1}(x,y)
\\
+u_{nk}\,\mathcal{P}_{n-1,k-1}(x,y)+v_{nk}\,\mathcal{P}_{n-1,k}(x,y)+w_{nk}\,\mathcal{P}_{n-1,k+1}(x,y),
\end{multline}
where the recurrence coefficients read
\begin{alignat*}{2}
a_{nk}&=\frac{(1-a q^{n-k+1})(1-abcq^{n+k+2})(1-d q^{n+1})}{(a b c q^{2n+2};q)_{2}},\quad &  b_{nk}&=1-a_{nk}-c_{nk},
\\
w_{nk}&=\frac{abc\sigma_k q^{n+2k+3}(abcq^{n+1}-d)(q^{n-k-1};q)_2}{(1-dq^{k+1})(a b c q^{2n+1};q)_2},\quad & f_{nk}&= a_{nk}(b c q^{k}\tau_{k}-\sigma_k+1),
\\
e_{nk}&=\frac{\tau_k bcq^{k}(dq^{k}-1)(1-dq^{n+1})(aq^{n-k+1};q)_2}{(abcq^{2n+2};q)_{2}},\quad & v_{nk}&=c_{nk}(bc q^{k}\tau_k-\sigma_k+1),
\\
g_{nk}&=\frac{\sigma_k (1-dq^{n+1})(a b c q^{n+k+2};q)_{2}}{(1-d q^{k+1})(abcq^{2n+2};q)_{2}},\quad &s_{nk}&=b_{nk}(bcq^{k}\tau_k-\sigma_k+1)+d(q^{k+1}\sigma_k-\tau_k),
\end{alignat*}
with
\begin{alignat*}{2}
t_{nk}&=\frac{q^{k+1}\sigma_{k}z_{n}(1-q^{n-k})(1-abcq^{n+k+2})}{1-d q^{k+1}},
\\
r_{nk}&=\tau_{k}z_{n}(dq^{k}-1)(1-aq^{n-k+1})(1-bcq^{n+k+1}),
\\
u_{nk}&=\frac{\tau_{k}a q^{n-k+1}(dq^{k}-1)(abcq^{n+1}-d)(bcq^{n+k};q)_{2}}{(abcq^{2n+1};q)_{2}},
\\
c_{nk}&=\frac{a d q^{n+1}(q^{n-k}-1)(1-b c q^{n+k+1})(1- a b c d^{-1} q^{n+1})}{(abcq^{2n+1};q)_{2}},
\end{alignat*}
and where $\sigma_k$, $\tau_k$ and $z_{n}$ are given by
\begin{align*}
\sigma_{k}&=\frac{(1-c q^{k+1})(1- b c q^{k+1})}{(bcq^{2k+1};q)_{2}},\quad 
\tau_{k}=-\frac{cq^{k+1}(1-q^{k})(1-bq^{k})}{(bcq^{2k};q)_{2}},
\\
z_{n}&=\frac{a b c q^{n+1}(1+q-dq^{n+1})-d}{(1-a b c q^{2n+1})(1-a b c q^{2n+3})}.
\end{align*}
\subsection{$\mathcal{J}_{n,k}(x,y)$ as a $q\rightarrow -1$ limit of $\mathcal{P}_{n,k}(x,y)$}
The two-variable Big $-1$ Jacobi polynomials $\mathcal{J}_{n,k}(x,y)$ can be obtained from the bivariate Big $q$-Jacobi by taking $q\rightarrow -1$. Indeed, a direct calculation using the expression \eqref{Biv-Q-Jacobi} shows that
\begin{align}
\label{Limit}
\lim_{\epsilon\rightarrow 0}\;\mathcal{P}_{n,k}(x,y;-e^{\epsilon \alpha},-e^{\epsilon \beta},-e^{\epsilon \gamma},\delta, -e^{\epsilon})=
\mathcal{I}_{n,k}(x,y; \alpha,\beta,\gamma,\delta),
\end{align}
where we have used the notation $\mathcal{I}_{n,k}(x,y; \alpha,\beta,\gamma,\delta)$ to exhibit the parameters appearing in the Big $-1$ Jacobi polynomials defined in \eqref{Poly}. A similar limit was considered in \cite{Vinet_Zhedanov_2012-07} to obtain the univariate Big $-1$ Jacobi polynomials in terms of the Big $q$-Jacobi polynomials.
\subsection{Eigenvalue equation for the Big $-1$ Jacobi polynomials}
The eigenvalue equation \eqref{Eigen-q} for the Big $q$-Jacobi polynomials and the relation \eqref{Limit} between the Big $q$-Jacobi polynomials and the Big $-1$ Jacobi polynomials can be used to obtain an eigenvalue equation for the latter.
\begin{proposition}
Let $L_{1}$ be the first-order differential/difference operator
\begin{multline}
\label{L1-Operator}
L_{1}=G_{5}(x,y)\,R_{x}R_{y}\partial_{y}+G_{6}(x,y)\,R_{y}\partial_{y}+G_{7}(x,y)\,R_{x}R_{y}\partial_{x}+G_{8}(x,y)\,R_{x}\partial_{x}
\\
+G_{1}(x,y)\,R_{x}R_{y}+G_2(x,y)\,R_{x}+G_{3}(x,y)\,R_{y}-(G_1(x,y)+G_2(x,y)+G_3(x,y))\,\mathbb{I},
\end{multline}
where $R_{x}, R_{y}$ are reflection operators and where the coefficients read
\begin{subequations}
\label{Coef}
\begin{alignat}{2}
G_{1}(x,y)&=\frac{x[1+\beta+\gamma-y(\alpha+\beta+\gamma+2)]-\delta[y(\alpha+\gamma+1)-\gamma]}{4xy},
\quad  &
G_{8}(x,y)&=\frac{(\delta+x)(x-y)}{2 x y},
\\
G_{2}(x,y)&=-\frac{x[x(\beta+\gamma+1)-\beta y]+\delta(y+\gamma x)}{4 x^2 y},
\quad &
G_{7}(x,y)&=\frac{(\delta+x)(y-1)}{2y},
\\
G_3(x,y)&=-\delta\left( \frac{x+\alpha xy-y[y(\alpha+\gamma+1)-\gamma]}{4x y^2}\right)
\quad &
G_{6}(x,y)&=\frac{\delta(x-y)(y-1)}{2 xy},
\\
G_{5}(x,y)&=\frac{(\delta+x)(y-1)}{2x}.
\end{alignat}
\end{subequations}
The Big $-1$ Jacobi polynomials satisfy the eigenvalue equation
\begin{align*}
L_{1}\,\mathcal{P}_{n,k}(x,y)=\mu_n \,\mathcal{P}_{n,k}(x,y),\qquad 
\mu_{n}=
\begin{cases}
-\frac{n}{2}, & \text{$n$ even},
\\
\frac{n+\alpha+\beta+\gamma+2}{2}, & \text{$n$ odd}.
\end{cases}
\end{align*}
Furthermore, let $L_2$ be the differential/difference operator
\begin{align*}
L_2=\frac{2(y-x)(x+\delta)}{x}R_{x}\partial_{x}+\frac{(\gamma+\beta+1)x^2+(\delta\gamma-\beta y)x+\delta y}{x^2}(R_{x}-\mathbb{I}),
\end{align*}
The Big $-1$ Jacobi polynomials satisfy the eigenvalue equation
\begin{align*}
L_2\,\mathcal{P}_{n,k}(x,y)=\nu_k \,\mathcal{P}_{n,k}(x,y),\qquad 
\nu_k=
\begin{cases}
2k,& \text{$k$ even},
\\
-2(k+\beta+\gamma+1), & \text{$k$ odd},
\end{cases}.
\end{align*}
\end{proposition}
\begin{proof}
The eigenvalue equation with respect to $L_1$ is obtained by dividing both sides of \eqref{Eigen-q} by $(1+q)$ and taking the $q\rightarrow -1$ limit according to \eqref{Limit}. The eigenvalue equation with respect to $L_2$ is obtained by combining \eqref{Poly}, \eqref{Jacobi} and \eqref{Eigen-Uni}.
\end{proof}
The two-variable Big $-1$ Jacobi polynomials are thus the joint eigenfunctions of the first order differential operators with reflections $L_1$ and $L_2$. It is directly verified that these operators  commute with one another, as should be. 

The $q\rightarrow -1$ limit \eqref{Limit} of the recurrence relations \eqref{Recurrence-q} can also be taken to obtain the recurrence relations satisfied by the Big $-1$ Jacobi polynomials. The result is as follows.
\begin{proposition}
The Big $-1$ Jacobi polynomials satisfy the recurrence relations
\begin{align*}
y\,\mathcal{J}_{n,k}(x,y)=\widetilde{a}_{nk}\,\mathcal{J}_{n+1,k}(x,y)+\widetilde{b}_{nk}\,\mathcal{J}_{n,k}(x,y)+\widetilde{c}_{nk}\,\mathcal{J}_{n-1,k}(x,y),
\end{align*}
\vspace{-2.5\baselineskip}
\begin{multline*}
x\,\mathcal{J}_{n,k}(x,y)=\widetilde{e}_{nk}\,\mathcal{J}_{n+1,k-1}(x,y)+\widetilde{f}_{nk}\,\mathcal{J}_{n+1,k}(x,y)+\widetilde{g}_{nk}\,\mathcal{J}_{n+1,k+1}(x,y)
\\ 
+\widetilde{r}_{nk}\,\mathcal{J}_{n,k-1}(x,y)+\widetilde{s}_{nk}\,\mathcal{J}_{n,k}(x,y)+\widetilde{t}_{nk}\,\mathcal{J}_{n,k+1}(x,y)
\\
+\widetilde{u}_{nk}\,\mathcal{J}_{n-1,k-1}(x,y)+\widetilde{v}_{nk}\,\mathcal{J}_{n-1,k}(x,y)+\widetilde{w}_{nk}\,\mathcal{J}_{n-1,k+1}(x,y).
\end{multline*}
With 
\begin{align*}
\widetilde{\tau}_k=\frac{k+\beta \phi_k}{2k+\beta+\gamma},\quad \widetilde{\sigma}_{k}=\frac{k+\beta\phi_k+\gamma+1}{2k+\beta+\gamma+2},\quad \widetilde{z}_{n}=\frac{(-1)^{n}-\delta(2n+\alpha+\beta+\gamma+2)}{(2n+\alpha+\beta+\gamma+1)(2n+\alpha+\beta+\gamma+3)},
\end{align*}
where $\phi_{k}=(1-(-1)^{k})/2$ is the characteristic function for odd numbers, the recurrence coefficients read
\begin{align*}
\widetilde{a}_{n,k}&=\frac{1+\delta_{n}}{2n+\alpha+\beta+\gamma+3}
\times
\begin{cases}
n-k+\alpha+1, & \text{$n+k$ even},
\\
n+k+\alpha+\beta+\gamma+2, & \text{$n+k$ odd},
\end{cases}
\\
\widetilde{c}_{n,k}&=\frac{1+\delta_{n+1}}{2n+\alpha+\beta+\gamma+1}
\times 
\begin{cases}
n-k, & \text{$n+k$ even},
\\
n+k+\beta+\gamma+1, &\text{$n+k$ odd},
\end{cases}
\\
\widetilde{e}_{n,k}&=\frac{\widetilde{\tau}_k(1-\delta_{k})(1+\delta_{n})}{2n+\alpha+\beta+\gamma+3}
\times
\begin{cases}
n-k+\alpha+1, & \text{$n+k$ even},
\\
n-k+\alpha+2, & \text{$n+k$ odd},
\end{cases}
\\
\widetilde{g}_{n,k}&=\frac{\widetilde{\sigma}_k(1+\delta_n)}{(1+\delta_k)(2n+\alpha+\beta+\gamma+3)}
\times
\begin{cases}
n+k+\alpha+\beta+\gamma+3, & \text{$n+k$ even},
\\
n+k+\alpha+\beta+\gamma+2, & \text{$n+k$ odd},
\end{cases}
\\
\widetilde{r}_{n,k}&=2\widetilde{\tau}_{k}\widetilde{z}_n((-1)^{k}-\delta)
\times
\begin{cases}
n-k+\alpha +1 & \text{$n+k$ even}
\\
n+k+\beta+\gamma+1 & \text{$n+k$ odd}
\end{cases}
\end{align*}
\begin{align*}
\widetilde{t}_{n,k}&=\frac{2(-1)^{k+1}\widetilde{\sigma}_k\widetilde{z}_n}{1+\delta_k}
\times 
\begin{cases}
n-k, & \text{$n+k$ even},
\\
n+k+\alpha+\beta+\gamma+2, & \text{$n+k$ odd},
\end{cases}
\\
\widetilde{u}_{nk}&=\frac{\widetilde{\tau}_k(1-\delta_k)(1-\delta_n)}{(2n+\alpha+\beta+\gamma+1)}
\times 
\begin{cases}
n+k+\beta+\gamma, & \text{$n+k$ even},
\\
n+k+\beta+\gamma+1, & \text{$n+k$ odd},
\end{cases}
\\
\widetilde{w}_{n,k}&=\frac{\widetilde{\sigma}_k(1-\delta_n)}{(1+\delta_k)(2n+\alpha+\beta+\gamma+1)}
\times 
\begin{cases}
n-k, & \text{$n+k$ even},
\\
n-k-1, & \text{$n+k$ odd},
\end{cases}
\end{align*}

with $\delta_{n}=(-1)^{n}\delta$ and
\begin{align*}
&\widetilde{b}_{n,k}=1-\widetilde{a}_{n,k}-\widetilde{c}_{n,k},\quad \widetilde{f}_{n,k}=\widetilde{a}_{n,k}(1-\widetilde{\sigma}_k-\widetilde{\tau}_{n,k}),\quad \widetilde{s}_{n,k}=\widetilde{b}_{n,k}(1-\widetilde{\sigma}_{k}-\widetilde{\tau}_{k})-\delta_{k} (\widetilde{\sigma}_{k}-\widetilde{\tau}_{k})
\\
& \widetilde{v}_{n,k}=\widetilde{c}_{n,k}(1-\widetilde{\sigma}_k-\widetilde{\tau}_{k})
\end{align*}
\end{proposition}
\begin{proof}
The result is obtained by applying the limit \eqref{Limit} to the recurrence relations \eqref{Recurrence-q}.
\end{proof}
\section{A Pearson-type system for Big $-1$ Jacobi}
Let us now show how the weight function $W(x,y)$ for the polynomials $\mathcal{J}_{n,k}(x,y)$ can be recovered as the symmetry factor for the operator $L_1$ given in \eqref{L1-Operator}. The symmetrization condition for $L_1$ is
\begin{align}
\label{Sym-Cond}
(W(x,y)L_1)^{*}=W(x,y)L_1,
\end{align}
where $M^{*}$ denotes the Lagrange adjoint. For an operator of the form
\begin{align*}
M=\sum_{\mu,\nu,k,j}A_{k,j}(x,y)\;\partial_{x}^{k}\;\partial_{y}^{j}\;R_{x}^{\mu}\;R_{y}^{\nu},
\end{align*}
for $\mu,\nu\in \{0,1\}$ and $k,j=0,1,2,\ldots$, the Lagrange adjoint reads
\begin{align*}
M^{*}=\sum_{\mu,\nu, k, j}(-1)^{k+j}\;R_{x}^{\mu}R_{y}^{\nu}\;\partial_{y}^{j}\;\partial_{x}^{k}\;A_{k,j}(x,y),
\end{align*}
where we have assumed that $W(x,y)$ is defined on a symmetric region with respect to $R_x$ and $R_y$.  Imposing the condition \eqref{Sym-Cond}, one finds the following system of Pearson-type equations:
\begin{subequations}
\begin{align}
\label{A}
W(x,y)\,G_{8}(x,y)&=W(-x,y)\,G_{8}(-x,y),
\\
\label{B}
W(x,y)\,G_{7}(x,y)&=W(-x,-y)\,G_{7}(-x,-y),
\\
W(x,y)\,G_{6}(x,y)&= W(x,-y)\,G_{6}(x,-y),
\\
W(x,y)\,G_{5}(x,y)&=W(-x,-y)\,G_{5}(-x,-y)
\\
\label{D}
W(x,y)\,G_{3}(x,y)&=W(x,-y)\,G_{3}(x,-y)-\partial_{y}(W(x,-y)\,G_{6}(x,-y)),
\\
\label{C}
W(x,y)\,G_{2}(x,y)&=W(-x,y)\, G_2(-x,y)-\partial_{x}(W(-x,y)\,G_{8}(-x,y)),
\\
\nonumber
W(x,y)\,G_{1}(x,y)&=W(-x,-y)\,G_1(-x,-y)
\\
&\qquad -\partial_{x}(W(-x,-y)\,G_{7}(-x,-y))-\partial_{y}(W(-x,-y)\,G_5(-x,-y)),
\end{align}
\end{subequations}
where the functions $G_{i}(x,y)$, $i=1,\ldots,8$, are given by \eqref{Coef}. We assume that $W(x,y)>0$ and moreover that $|\delta|\leqslant |x|\leqslant y\leqslant 1$. Upon substituting \eqref{Coef} in \eqref{A}, one finds
\begin{align*}
(x+\delta)\,(x-y)\,W(x,y)=-(x-\delta)\,(x+y)\,W(-x,y),
\end{align*}
for which the general solution is of the form
\begin{align}
\label{Inter-1}
W(x,y)=\theta(x)\,(x-\delta)\,(x+y)\,f_1(x^2,y),
\end{align}
where $f_1$ is an arbitrary function. Using \eqref{Inter-1} in \eqref{B} yields
\begin{align*}
(y-1)\,f_1(x^2,y)=(y+1)\,f_1(x^2,-y),
\end{align*}
which has the general solution
\begin{align*}
f_1(x^2,y)=\theta(y)\,(y+1)\,f_2(x^2,y^2),
\end{align*}
where $f_2$ is an arbitrary function. The function $W(x,y)$ is thus of the general form
\begin{align}
\label{Dompe}
W(x,y)=\theta(x)\theta(y)\,(x-\delta)\,(x+y)\,(y+1)\,f_2(x^2,y^2).
\end{align}
Upon substituting the above expression in \eqref{C}, one finds after simplifications
\begin{align*}
\left[\frac{\beta-1}{x^2-\delta^2}+\frac{\gamma-1}{x^2-y^2}\right]\,x\,f_2(x^2,y^2)=\partial_{x}\,f_2(x^2,y^2).
\end{align*}
After separation of variable, the result is
\begin{align*}
f_2(x^2,y^2)=(x^2-\delta^2)^{\frac{\beta-1}{2}}(y^2-x^2)^{\frac{\gamma-1}{2}}f_3(y^2).
\end{align*}
Finally, upon substituting the above equation in \eqref{D} one finds
\begin{align*}
y(\alpha-1)f_3(y^2)=(y^2-1)\partial_{y}f_3(y^2),
\end{align*}
which gives $f_3=(1-y^2)^{\frac{\alpha-1}{2}}$ and thus
\begin{align*}
f_2(x^2,y^2)=(x^2-\delta^2)^{\frac{\beta-1}{2}}(y^2-x^2)^{\frac{\gamma-1}{2}}(1-y^2)^{\frac{\alpha-1}{2}}.
\end{align*}
Upon combining the above expression with \eqref{Dompe}, we find
\begin{align*}
W(x,y)=\theta(x\,y)|y|^{\beta+\gamma}(1+y)\left(1+\frac{x}{y}\right)\left(\frac{x-\delta}{y}\right)(1-y^2)^{\frac{\alpha-1}{2}}\left(1-\frac{x^2}{y^2}\right)^{\frac{\gamma-1}{2}}\left(\frac{x^2-\delta^2}{y^2}\right)^{\frac{\beta-1}{2}},
\end{align*}
which indeed corresponds to the weight function of proposition 2.1. The weight function $W(x,y)$ for the two-variable Big $-1$ Jacobi polynomials thus corresponds to the symmetry factor for $L_1$.
\section{Conclusion}
In this paper, we have introduced and characterized a new family of two-variable orthogonal polynomials that generalize of the Big $-1$ Jacobi polynomials. We have constructed their orthogonality measure and we have derived explicitly their bispectral properties. We have furthermore shown that the weight function for these two-variable polynomials can be recovered by symmetrization of the first-order differential operator with reflections that these polynomials diagonalize. The two-variable orthogonal polynomials introduced here are the first example of the multivariate generalization of the Bannai-Ito scheme. It would be of great interest to construct multivariate extensions of the other families of polynomials  of this scheme.
\section*{Acknowledgments}
VXG holds an Alexander-Graham-Bell fellowship from the Natural Science and Engineering Research Council of Canada (NSERC). JML holds a scholarship from NSERC. The research of LV is supported in part by NSERC. AZ would like to thank the Centre de recherches math\'ematiques for its hospitality.
\begin{appendices}
\section{Normalization coefficients for Big $-1$ Jacobi polynomials}
In this appendix, we present a derivation of the normalization coefficients \eqref{hn} and  \eqref{Hn-Tilde} appearing in the orthogonality relation of the univariate Big $-1$ Jacobi polynomials. The result is obtained by using their kernel partners, the Chihara polynomials. These polynomials, denoted by $C_{n}(x;\alpha,\beta,\gamma)$, have the expression \cite{Genest-2013-09-02}
\begin{align*}
&C_{2n}\left(x;\alpha,\beta,\gamma\right)= (-1)^n\frac{(\alpha+1)_n}{(n+\alpha+\beta+1)_n}\;\pfq{2}{1}{-n, n+\alpha+\beta+1}{\alpha+1}{x^2-\gamma^2} ,
\\
&C_{2n+1}\left(x;\alpha,\beta,\gamma\right)= (-1)^n\frac{(\alpha+2)_n}{(n+\alpha+\beta+2)_n}(x-\gamma)\;\pfq{2}{1}{-n, n+\alpha+\beta+2}{\alpha+2}{x^2-\gamma^2}.
\end{align*}
For $\alpha,\beta>-1$, they satisfy the orthogonality relation
\begin{align}
\label{Ortho-A}
\int_{\mathcal{E}} C_{n}\left(x;\alpha,\beta,\gamma\right)C_{m}\left(x;\alpha,\beta,\gamma\right)\theta(x)(x+\gamma)(x^2-\gamma^2)^\alpha(1+\gamma^2-x^2)^{\beta}\;\mathrm{d}x=\eta_n\delta_{nm},
\end{align}
on the interval $\mathcal{E}=[-\sqrt{1+\gamma^2},-|\gamma|]\cup[|\gamma|,\sqrt{1+\gamma^2}]$ and their normalization coefficients read
\begin{align*}
\eta_{2n}=\frac{\Gamma(n+\alpha+1)\Gamma(n+\beta+1)}{\Gamma(n+\alpha+\beta+1)}\frac{n!}{(2n+\alpha+\beta+1)[(n+\alpha+\beta+1)_n]^2},
\\
\eta_{2n+1}=\frac{\Gamma(n+\alpha+2)\Gamma(n+\beta+1)}{\Gamma(n+\alpha+\beta+2)}\frac{n!}{(2n+\alpha+\beta+2)[(n+\alpha+\beta+2)_n]^2}.
\end{align*}
Let $\widehat{J}_{n}(x)$ be the monic Big $-1$ Jacobi polynomials:
\begin{align*}
\widehat{J}_n(x)= \kappa_n J_n(x) = x^n+\mathcal{O}(x^{n-1}),
\end{align*}
where $\kappa_n$ is given by
\begin{gather*} 
\kappa_n(a,b,c)=\begin{cases}
\qquad \frac{(1-c^2)^{\frac{n}{2}}\left(\frac{a+1}{2}\right)_{\frac{n}{2}}}{\left(\frac{n+a+b+2}{2}\right)_{\frac{n}{2}}} ,                  &\text{$n$ even,}
\\
\frac{(1+c)(1-c^2)^{\frac{n-1}{2}}\left(\frac{a+1}{2}\right)_{\frac{n+1}{2}}}{\left(\frac{n+a+b+1}{2}\right)_{\frac{n+1}{2}}},   &\text{$n$ odd.} \end{cases}
\end{gather*} 
The kernel polynomials $\widetilde{K}_n(x;a,b,c)$ associated to $\widehat{J}_{n}(x;a,b,c)$ are defined through the Christoffel transformation \cite{Chihara-2011}:
\begin{align}
\label{Chris}
\widetilde{K}_n(x;a,b,c)=\frac{\widehat{J}_{n+1}(x)-\frac{\widehat{J}_{n+1}(\nu)}{\widehat{J}_n(\nu)}\widehat{J}_n(x)}{x-\nu}.
\end{align}
For $\nu=1$, the monic polynomials $\widetilde{K}_n(x;a,b,c)$ can be expressed in terms of the Chihara polynomials. Indeed, one can show that \cite{Genest-2013-09-02}
\begin{align}
\label{Relation}
\widetilde{K}_n(x;a,b,c)=(\sqrt{1-c^2})^{n}\,C_n\left(\frac{x}{\sqrt{1-c^2}};\frac{b-1}{2},\frac{a+1}{2},\frac{-c}{\sqrt{1-c^2}}\right).
\end{align}
Let $\mathcal{M}$ be the orthogonality functional for the Big $-1$ Jacobi polynomials. In view of the relation \eqref{Chris}, one can write \cite{Chihara-2011}
\begin{align*}
 \mathcal{M}[(x-\nu)\;\widetilde{K}_n(x;a,b,c)\;x^k]=-\frac{\widehat{J}_{n+1}(\nu)}{\widehat{J}_{n}(\nu)}\mathcal{M}[\widehat{J}^2_n(x;a,b,c)]\,\delta_{kn}.
\end{align*}
By linearity, one thus finds
\begin{align*}
 \widetilde{\eta}_n=-\frac{\widehat{J}_{n+1}(\nu)}{\widehat{J}_{n}(\nu)}\hat{h}_n.
\end{align*}
where $\widetilde{\eta}_n=\mathcal{M}[(x-\nu)\;\widetilde{K}_n^2(x)]$ and $ \hat{h}_n=\mathcal{M}[\widehat{P}_n^2(x)]$. For $\nu=1$, the value of $\widetilde{\eta}_n$ is easily computed from  \eqref{Ortho-A} and \eqref{Relation}. The above relation then gives the desired coefficients.
\end{appendices}

\footnotesize
\begin{multicols}{2}

\end{multicols}
\end{document}